  \documentclass{amsart}
\usepackage{amsmath,amssymb,amsthm,amscd, amsxtra, amsfonts}

\newtheorem{theorem}{Theorem}[section]

\newtheorem{lemma}{Lemma}[section]

\newtheorem{remark}{Remark}[section]

\newtheorem{proposition}{Proposition}[section]

\def \a{\alpha }

\def \l{\lambda }

\begin{document}

\newcommand{\wta}{{\rm {wt} }  a }
\newcommand{\R}{\frak R}
\newcommand{\cV}{\mathcal V}
\newcommand{\cA}{\mathcal A}
\newcommand{\cL}{\mathcal L}
\newcommand{\J}{\mathcal H}
\newcommand{\G}{\mathcal G}
\newcommand{\wtb}{{\rm {wt} }  b }
\newcommand{\bea}{\begin{eqnarray}}
\newcommand{\eea}{\end{eqnarray}}
\newcommand{\be}{\begin {equation}}
\newcommand{\ee}{\end{equation}}
\newcommand{\g}{\frak g}
\newcommand{\hg}{\hat {\frak g} }
\newcommand{\hn}{\hat {\frak n} }
\newcommand{\h}{\frak h}
\newcommand{\U}{\mathcal U}
\newcommand{\hh}{\hat {\frak h} }
\newcommand{\n}{\frak n}
\newcommand{\Z}{\Bbb Z}
\newcommand{\N}{{\Bbb Z} _{> 0} }
\newcommand{\Zp} {\Z _ {\ge 0} }
\newcommand{\C}{\Bbb C}
\newcommand{\Q}{\Bbb Q}
\newcommand{\1}{\bf 1}
\newcommand{\la}{\langle}
\newcommand{\ra}{\rangle}
\newcommand{\NS}{\bf{ns} }

\newcommand{\hf}{\mbox{$\tfrac{1}{2}$}}
\newcommand{\thf}{\mbox{$\tfrac{3}{2}$}}

\newcommand{\W}{\mathcal{W}}
\newcommand{\non}{\nonumber}
\def \l {\lambda}
\baselineskip=14pt
\newenvironment{demo}[1]%
{\vskip-\lastskip\medskip
  \noindent
  {\em #1.}\enspace
  }%
{\qed\par\medskip
  }

\def \l {\lambda}
\def \a {\alpha}

\title[Affine vertex algebra associated to $\frak{gl}(1 \vert 1)$   at the critical level  ]{A note on the  affine vertex algebra associated to $\frak{gl}(1 \vert 1)$   at the critical level and its generalizations  }
\author{ Dra\v zen Adamovi\' c}
\address{Department of Mathematics,  Faculty of Science, University of Zagreb, Croatia}
\email{adamovic@math.hr}

\keywords{vertex algebras, affine Lie superalgebras, critical level, $\mathcal{W}$-algebras}

\subjclass{17B69; 17B67}

\begin{abstract}
In this note we present an explicit realization of the affine vertex algebra $V^{cri}(\frak{gl}(1 \vert 1)) $ inside of the tensor product $F\otimes M$ where $F$ is a fermionic verex algebra and $M$ is a commutative vertex algebra. This immediately gives an alternative description of the center of $V^{cri}(\frak{gl}(1 \vert 1) ) )$  as a subalgebra  $M _ 0$ of $M$. We reconstruct the Molev-Mukhin formula for the Hilbert-Poincare series of  the center of $V^ {cri}(\frak{gl}(1 \vert 1) )$. Moreover, we construct a family of irreducible $V^{cri}(\frak{gl}(1 \vert 1))$ --modules realized on $F$ and parameterized by $\chi^+, \chi ^- \in {\C}((z)). $ We propose a generalization of $V^ {cri}(\frak{gl}(1 \vert 1))$ as a critical level version of the super $\mathcal W_{1+\infty}$ vertex algebra.
\end{abstract}

  \maketitle

 \centerline{\it Dedicated to the memory of Sibe Marde\v si\' c}
\def \l {\lambda}
\def \a {\alpha}

 \section{Introduction}

 Let $\frak g$ be the Lie superalgebra, and $  {\hg} = \frak g \otimes {\C}[t,t^{-1}] + {\C} K $ the associated affine Lie superalgebra.  The representation theory of affine Lie superalgebras at certain level $k$ is closely related with the representation theory of the universal affine vertex algebra $V^k(\g)$ associated to $\g$.     
 
 For a given vertex algebra $V$, it is important problem to describe the structure of the center of $V$. The center is defined as the following vertex subalgebra of $V$: 
 $ {\mathfrak Z}(V) = \{ v \in V \ \vert \ [Y(v,z), Y(w,z)] = 0 \ \forall w \in V\}. $
 
 In the case of affine vertex algebras, the center is non-trivial only in the case of critical level. We denote the universal affine vertex algebra at the critical level by $V^{cri}(\g)$ and its center by ${\mathfrak Z}(\widehat {\g})$. When $\g$ is a simple Lie algebra,  the center  of $V^{cri}(\g)$, called the Feigin-Frenkel center,  is finitely-generated commutative  vertex algebra (see \cite{Fr},  \cite{FB}, \cite{FF-center}). 
 
 But in the case of Lie superalgebras, ${\mathfrak Z}(\widehat \g)$ can be infinitely--generated and it is not completely understood yet. 
 
 In the recent paper \cite{MM},  A.  Molev and E. Mukhin  determined  the structure of ${\mathfrak Z}(\g)$ in the case $\g = {\frak gl}(1 \vert 1)$, and presented conjectures on  the structure of  the center for $\g={\frak gl}(m \vert n)$. In the case $\g = {\frak gl}(1 \vert 1)$, they constructed  a family  of central, Sugawara  elements (see also \cite{MR}) and proved that these central elements generate ${\mathfrak Z}(\g)$ (see \cite[Theorem 2.1]{MM}).
 
 The aim of this note is to present an alternative construction of the center of $V^{cri}(\g)$ in the case $\g = {\frak gl}(1 \vert 1)$ by using an explicit free-field realization of $V^{cri}(\g)$. Let us describe our result.  We consider the commutative vertex algebra
 $$M = {\C}[a^{\pm} (m+1/2) \ \vert m \in {\Z}_{< 0}]$$ which is uniquely determined by the following commutative fields:
 $$ a^{\pm} (z) =\sum_{m <0} a^{\pm} (m+1/2) z^{-m-1}. $$
 We prove:
 \begin{theorem} ${\mathfrak Z}({\widehat \g})$ is isomorphic to a vertex subalgebra of $M$ generated by the following fields
 $$ a^+(z)  \  \partial ^ k a^-(z ), \quad k \ge 0. $$
 In particular, the Hilbert Poincare series of ${\mathfrak Z}({\widehat \g})$  coincides with the $q$-character of the simple vertex algebra ${\mathcal W}_{1+\infty,c}$ with $c=-1$.
 \end{theorem}
 
 In Section \ref{Whittaker}, we explicitly construct  a large family of the irreducible $V^{cri}(\frak{gl}(1 \vert 1))$ modules parametrized by
 $ \chi^{\pm}(z) \in {\C}((z)). $ It is interesting that these modules are realized  on the Clifford vertex algebra $F$.
 
 It will be  interesting to study generalization of our description of  ${\mathfrak Z}({\widehat \g})$ for $\g =\frak{gl}(1 \vert 1)$  in various directions. Of course, the most important open question is the determination of the center in the case  $\g={\frak gl}(m \vert n)$. It seems that the general case is much more complicated.
 But in the present paper we propose one different generalization.
 In Section \ref{generalization} we introduce  a vertex algebra $V_n$ having large center, such that $V_1= V^{cri}(\frak{gl}(1 \vert 1))$. The construction of $V_n$ is a critical level version  of the super ${\mathcal W}_{1+\infty}$--algebra, constructed and analysed  by T. Creutzig and A. Linshaw in \cite{CL}.
 
 In our forthcoming papers we shall study these generalizations in more details.
  \vskip 5mm
 
 {\bf Acknowledgment.} We  would like to thank Alexander  Molev and Antun Milas  for  useful discussions.
  The author is partially supported by the Croatian Science Foundation under the project 2634
and by the Croatian Scientific Centre of Excellence QuantiXLie.
\vskip 5mm

 \section{Certain vertex algebras}
\label{certain}
  
In  the  paper  we assume that the reader is familiar with the basic theory of vertex algebras. In this section we briefly review  results on  vertex algebras which we use in the paper.
 
\subsection{ The affine vertex algebra $V^{cri} (\mathfrak{gl}(m \vert n) )$ } 
The vertex algebras associated to affine Lie superalgebras are most important examples of vertex algebras (cf. \cite{FB}, \cite{K}, \cite{LL}).

We recall the definition of universal affine vertex algebra $V^{cri} (\mathfrak{gl}(m \vert n) )$ following notations in \cite{MM}.

For $1 \le j \le m+n$ we define $\bar j \in \Z_2$:
$$\bar j = \bar 0  \quad  \mbox{for} \  1 \le j \le m, \qquad   \bar j = 1  \quad \mbox{for} \  m+1 \le j \le m+n. $$

Let $\g$ be Lie superalgebra ${\mathfrak gl}(m \vert n)$ with standard basis $E_{i, j}$, $1 \le i, j \le n+m$. The affine Lie superalgebra $\hg = \g \otimes {\C}[t, t^{-1}] + {\C} K $ with commutation relations
\bea
[E_{i,j} (r), E_{k,\ell }(s) ] =&& \delta_{k,j} E_{i,\ell}(r+s) -\delta_{i, \ell}  E_{k,j} (r+s) (-1) ^{( \bar i + \bar j)(\bar k + \bar \ell) } \nonumber \\
&& + K \left ( (n-m) \delta_{k,j} \delta_{i,\ell} (-1) ^{\bar i} + \delta_{i,j} \delta_{k,\ell} (-1) ^{\bar i + \bar k} \right) r \delta_{r+s,0} \label{com}
\eea
where $K$ is even central element, and $ x(n) = x\otimes t^r$ for $x \in \g$, $r, s \in {\Z}$.

Let ${\C} v$ be the $1$--dimensional $P= \g \otimes {\C}[t] + {\C} K$--module such that
 $$K v = v, \quad  (\g \otimes {\C}[t]  )v = 0. $$Then the $\hg$--module
$$V^{cri} (\g) = U(\hg) \otimes _{U(P)} {\C} v$$
has the vertex algebra structure which is uniquely generated by the fields
$$x (z) = \sum_{ r\in {\Z} } x(r) z^{-r-1} \quad (x \in {\g}).$$
The vertex algebra $V^{cri} (\g)$ is called the universal affine vertex algebra at the critical level.

The center of $V^{cri} (\g)$ has the following description
$$ {\mathfrak Z}(\hg) = \{ w \in V^{cri} (\g) \ \vert ( \g \otimes {\C}[t]) . w = 0\}. $$
 
\begin{remark}
The standard choice of central element in $\hg$ is $C= (n-m) K$ and it acts on $V^{cri}(\g)$ as $-h^{\vee} \mbox{Id}$. Moreover, if $m \ne n$, then the standard definition of the affine  vertex algebra $V^{-h^{\vee}} (\g)$  at the critical level  coincides with $V^{cri}(\g)$. But the case $m=n$ is different. It was observed in \cite{MM} and \cite{MR} that then one needs to take central element  $K$  acting  as identity on $V^{cri}(\g)$.
\end{remark}
\subsection{ Clifford vertex algebras }

  The Clifford
algebra $CL_n  $ is a complex associative algebra generated by
$$ \Psi_i ^{\pm}(r) , \  r  \in \hf + {\Z}, \  1\le i  \le n, $$ and relations
\bea
&& \{\Psi_i  ^{\pm}(r) , \Psi_j ^{\mp}(s) \} = \delta_{i,j} \delta_{r+s,0}; \quad
 \{\Psi_i ^{\pm}(r) , \Psi_ j ^{\pm}(s)\}=0
\nonumber
\eea
where $r, s \in \frac{1}{2} + {\Z}$.

 Let $F^{(n)} $ be the irreducible $CL $--module generated by
 the
cyclic vector ${\1}$ such that
$$ \Psi_i ^{\pm} (r) {\1} = 0 \quad
\mbox{for} \ \ r > 0 .$$

As a vector space, $F^{(n)} $ is isomorphic to the exterior algebra
$$F^{(n)} \cong \bigwedge \left(   \Psi_i ^{\pm }({-m-{\hf}}) \ \vert m \in {\Z}_{\ge 0}, \ i=1, \dots, n \right). $$

Define the following   fields on $F^{(n)} $ 
$$   \Psi_i ^{+}(z) = \sum_{ m \in   {\Z}
 } \Psi_i ^{+}(m+{\hf} )  z ^{-m- 1}, \quad  \Psi_i ^{-} (z) = \sum_{ m \in {\Z}
 } \Psi_i ^{-} (m+{\hf} )  z ^{-m-1}.$$

 The fields $\Psi_i ^{+}(z)$ and $\Psi_i^{-}(z)$, $i=1, \dots, n$ generate on $F^{(n)} $    the
unique structure of a simple vertex algebra
(cf.  \cite{FB}, \cite{K}).

It is well-known that the  vertex subalgebra of $F^{(n)}$ generated by the vectors 
$$\{ e_{i,j} = \Psi_i ^+ (-\frac{1}{2}) \Psi_j ^- (-\frac{1}{2}) {\bf 1}  \ \vert \ i,j = 1, \dots, n \} $$ is isomorphic  to the simple affine vertex algebra
$V_1 ({\frak gl}_n) $ at level $1$. This implies that the Lie algebra ${\mathfrak gl}_n$ acts on $F^{(n)} $ by derivations. The fixed point subalgebra
$$(F^{(n)} )^{{\frak gl}_n}$$ is isomorphic to the simple vertex algebra $\mathcal W_{1+ \infty}$ at central charge $c=n$
  (see \cite{FKRW}  for details on this construction).

Let us consider the case $n=1$. Set $F:= F^{(1)}$, $\Psi^{\pm} (z) = \Psi_1 ^{\pm}$.

A basis of $F $ is given by

\bea \label{bas-fer} \Psi^{+ }({-n_1-{\hf}})  \cdots \Psi^{+}({-n_r-{\hf}})  \Psi^{-}({-k_1-{\hf}})  \cdots \Psi^{-}({-k_s-{\hf}})
 {\1} \eea

where $n_i, k_i \in {\Zp}$,  $n_1 >n_2 >\cdots >n_r  $, $k_1 >k_2
>\cdots
>k_s $.

  Let $\alpha:= \Psi^{+}(-1/2) \Psi^{-}(-1/2) {\bf 1} $. Then the operator $\alpha(0)$ defines on $F$   the following $\Z$--gradation 
$$ F =\bigoplus _{\ell  \in \Z} F_{\ell}, $$
where 
$$F_ {\ell}=\{ v \in F \ \vert \ \alpha(0) v = \ell v \}. $$

A basis of $F_ {\ell}$ is given by vectors (\ref{bas-fer}) such that $\ell = r -s$.



Recall here that by using a  boson--fermion correspondence, the  fermionic vertex algebra $F$ can be realized as the lattice vertex algebra $V_{L} = M(1) \otimes {\C}[L]$ 
where 
$$ L= {\Z} \alpha, \quad \la \alpha, \alpha \ra = 1, $$
$M(1)$ is the  Heisenberg vertex algebra generated by $\alpha$ and ${\C}[L]$ is the  group algebra of $L$. In particular, as $M(1)$--modules
$  F_ {\ell} \cong M(1). e ^{\ell \alpha}$.

\subsection{Weyl vertex algebra  $W^{(n)}$ }

The Weyl vertex algebra $W^{(n)}$ is generated by the fields
$${\gamma}_i ^{\pm} (z) = \sum_{ m  \in \Z} {\gamma}_i^{\pm} (m +\tfrac{1}{2}) z ^{-m-1}, $$
whose components satisfy the commutation relation for Weyl algebra
$$[\gamma_i^+ (r), \gamma_j ^- (s) ] =\delta_{i,j} \delta_{r+s,0}, \quad  [\gamma_i ^{\pm}(r) , \gamma_j^{\pm}(s)]=0 \quad (r,s \in \tfrac{1}{2} + \Z, \ i, j = 1, \dots, n). $$

Choose the following Virasoro vector of central charge $c=-n$:
$$\omega = \frac{1}{2} \sum_{i=1}^n \left( \gamma_i ^{-} (-\tfrac{3}{2})  \gamma_i ^{+} (-\tfrac{1}{2}) -    \gamma_i ^{+} (-\tfrac{3}{2})  \gamma_i ^{-} (-\tfrac{1}{2}) \right) {\bf 1}.   $$
The associated Virasoro field is given by $$L(z) = \frac{1}{2}  \sum_{i}^n  \left( :  (\partial \gamma _i ^- (z) ) \gamma_i ^+ (z):  - :(\partial \gamma _i ^+ (z) )  \gamma_i ^-  (z) :\right) = \sum_{m=-\infty} ^{\infty} L(m) z^{-m-2}. $$

The  vertex subalgebra of $W^{(n)}$  generated by the vectors 
$$\{ e_{i,j} = -\gamma_i ^+ (-\frac{1}{2}) \gamma_j ^- (-\frac{1}{2}) {\bf 1}  \ \vert \ i,j = 1, \dots, n \} $$ is isomorphic  to the simple affine vertex algebra
$V_{-1} ({\frak gl}_n) $ at level $-1$. 
This realization was found by A. Feingold and I. Frenkel in \cite{FF}, and the simplicity  of the realization  (for $n \ge 3$)  was proved by the author and O. Per\v se in \cite{AP-2014}.

This again  implies that the Lie algebra ${\frak gl}_n$ acts on $W^{(n)} $ by derivations. The fixed point subalgebra
$$(W^{(n)} )^{{\frak gl}_n}$$ is isomorphic to the simple vertex algebra $\mathcal W_{1+ \infty}$ at central charge $c=-n$
  (for details see  \cite{KR}). We denote this vertex algebra by $\mathcal W_{1+ \infty, -n}$. 
  
  The case $n=1$ was studied by W. Wang in \cite{Wa}. It was proved in \cite[Section 5]{Wa} that  the
  $q$--character  (i.e., the    Hilbert--Poincare series ) of  $\mathcal W_{1+ \infty, -1}$ with respect to $L(0)$ is given by
  
  \bea  &&  \mbox{ch}  [\mathcal W_{1+ \infty, -1} ](q) =  \mbox{tr} \ q^{L(0)} \vert _{ \mathcal W_{1+ \infty, -1}} \nonumber \\ = &&   \mbox{Res}_z  z^{-1} \frac{1}{ \prod_{k \ge 1} (1 -q^{k-1/2} z) (1-q^{k-1/2} z^{-1})}  \nonumber  \\ =&& \label{HP1}  \frac{1}{(q)_{\infty} ^ 2 }\sum_{n \in \Z } \mbox{sign}  (n)  q^ { 2 n ^2 +n } =\frac{1}{(q)_{\infty}  ^2 }\sum_{k=0} ^ {\infty} (-1) ^ k  q^ {\frac {k^2 + k}{2} }. \eea

By  using \cite[Proposition 3.3]{MM} or \cite[Example 7.1]{BM}  we see that (\ref{HP1}) can be written as

\bea \label{HP2}    \frac{1}{(q)_{\infty} ^ 2 }\sum_{n \in \Z } \mbox{sign}  (n)  q^ { 2 n ^2 +n }    =\frac{1}{(q)_{\infty}   }\sum_{k=0} ^ {\infty}   \frac{q^ {k^2 + k }}{(q)_{k} ^ 2  }. \eea

In the above formulas we use standard notations:
$$ (q)_{\infty} =\prod_{i \ge 1} (1-q^i), \quad (q)_{k} =\prod_{i=1} ^k  (1-q^i). $$

\begin{remark} 
We should mention that  the identity   (\ref{HP2})  was first proved by Ramanujan.
Let $\mathcal W (2,2p-1)$ denotes the singlet vertex algebra with central charge $c_{1,p} = 1 - 6 (p-1) ^2 / p$ (cf. \cite{AM-2007}). 
Note that  $\mathcal W_{1+ \infty, -1} $ is isomorphic to the tensor product of a  rank-one Heisenberg vertex algebra  and the singlet vertex algebra $\mathcal W(2,2p-1)$  for $p=2$. So the combinatorial identity  (\ref{HP2}) is essentially the $q$--series identity of the  vacuum character  for $\mathcal W(2,2p-1)$ for $p=2$.

An interesting generalization  of the identity (\ref{HP2}) for $\mathcal W(2,2p-1)$-- characters  for $p \ge 3$ was proved by K. Bringmann and A. Milas in \cite[Proposition 7.2]{BM}.
\end{remark}

\subsection{Commutative vertex algebra $M^{(n)}$ }

 Let $$M^{(n)} = {\C}[{a}_i ^{+}(m +\frac{1}{2} ), {a}_i^{-}(m+\frac{1}{2}) \  \vert \ m \in {\Z}_{ <0} , i = 1, \dots, n
]$$ be the commutative vertex algebra generated by the fields
$${a}_i ^{\pm} (z) = \sum_{ m \in {\Z}_{ <0} } {a}_i ^{\pm} (m+\frac{1}{2}) z ^{-m-1}. $$
The $\frac{1}{2} {\Z}_{\ge 0}$--grading on $M^{(n)}$ is given by the operator $d \in \mbox{End} (M^{(n)})$ which is uniquely determined by the following formula:
$$ [d,  {a}_i ^{\pm} (m+\frac{1}{2}) ] = -  (m+\frac{1}{2}) {a}_i ^{\pm} (m+\frac{1}{2}). $$

The Lie algebra ${\frak gl}_n$ acts on the vertex algebra $M^{(n)}$ by derivations. This action is uniquely determined by the following formula:
$$e_{i,j} . a_k ^{+}  (r) = \delta_{j,k}   a_i ^+ (r) , \quad e_{i,j} . a_k ^{-} (r)  = -\delta_{ i ,k} a_j ^-  (r),  $$
for $1 \le i, j \le n$, $r \in \frac{1}{2} + {\Z}$.

Moreover $( M^{(n)} )^{{ \frak gl}_n}$ is a  subalgebra of  $M^{(n)}  $. 
 $( M^{(n)} )^{{ \frak gl}_n}$  is $d$--invariant, and the operator $d$ defines ${\Z}_{\ge 0}$--grading in $( M^{(n)} )^{{ \frak gl}_n}$.

In the case $n=1$ we set $a ^{\pm} (z) := a_1 ^{\pm} (z)$. The vertex algebra $M = M^{(1)} $ has the following $\Z$--gradation 

$$ M =\bigoplus _{\ell  \in \Z} M_{\ell}, $$
where 
$M_{\ell}$ is a linear span of vectors

\bea \label{bas-com} a^{+ }({-n_1-{\hf}})  \cdots a^{+}({-n_r-{\hf}})  a^{-}({-k_1-{\hf}})  \cdots a^{-}({-k_s-{\hf}})
 {\1} \eea
such that  $n_i, k_i \in {\Zp}$,  $n_1 >n_2 >\cdots >n_r  $, $k_1 >k_2
>\cdots
>k_s $ and $\ell = r-s$.

It is easy to see that $M_0$ is a vertex subalgebra of $M$ and  that $M _0$  is  strongly generated by vectors
\bea   &&  a^+ (-1/2) a^ -(-m-1/2) {\bf 1} \qquad (m \in \Zp). \label{generators} \eea

 \begin{remark}
By using construction  from  \cite{L-JEMS} one can show  a  more general  result that    $(M^{(n)} )^{{ \frak gl}_n}$    is strongly generated by vectors
$$ \sum_{i=1} ^n  a_i ^+ ( -1/2) a_i ^ -(-m-1/2) {\bf 1} \qquad (m \in \Zp). $$
\end{remark}

$M_0$ is a graded commutative vertex algebra with the following $\Z_{\ge 0}$--gradation:
$$M_0 = \bigoplus_{m = 0} ^{\infty} M_0(m) \quad M_0(m) = \{ v \in M_0(m) \ d v = m v\}. $$
 The Hilbert--Poincare series of $M_0$ is
\bea
&& \sum_{m=0} ^{\infty} \dim M_0 (m) q^m \nonumber \\
= && \mbox{Res}_z  z^{-1} \frac{1}{ \prod_{k \ge 1} (1 -q^{k-1/2} z) (1-q^{k-1/2} z^{-1})}  \nonumber \\
= &&  \frac{1}{(q)_{\infty}   }\sum_{k=0} ^ {\infty}   \frac{q^ {k^2 + k }}{(q)_{k} ^ 2  }. \nonumber 
\eea

Therefore $M_0$ is a graded commutative vertex algebra whose Hilbert--Poincare series is given by (\ref{HP1}).

 Let $\chi^{\pm}(z) = \sum_{n \in {\Z} }
\chi^{\pm}_{n} z^{-n-1} \in {\C} ((z))$. Let $M(\chi^{+},
\chi^{-})$ denotes the $1$--dimensional irreducible $M$--module
with the property that every element ${a}^{\pm}(n)$ acts on
$M(\chi^{+}, \chi^{-})$ as multiplication by $\chi^{\pm}_n \in
{\C}$.
       
   \section{The  main  result }
   
   \label{new-real}
   In this section we shall present an explicit realization of the vertex algebra $V^{cri}(\mathfrak{gl} (1 \vert 1))$.  Our construction  is similar to the realization of affine   $\mathfrak{gl} (1 \vert 1)$ from \cite{K}. Main difference is that we replace Weyl vertex algebra   (also called the symplectic boson vertex algebra) by the commutative vertex algebra $M$.

   We  consider the following subalgebra of $F  \otimes M$:
   
   \bea \label{defin-voa} V= (F  \otimes M) _0 = \bigoplus_{\ell \in \Z}  F_ {\ell} \otimes M_{-\ell} .\eea
   
   Let us denote by ${\mathfrak Z}(V)$ the center of the vertex algebra $V$.
   \begin{lemma} \label{lema-center}
 ${\mathfrak Z}(V)$ is  isomorphic to the vertex algebra $M_0$.
   \end{lemma}
   \begin{proof}
   First we notice that $M_0 \subset {\mathfrak Z}(V)$. Since  $\alpha \in  V$, we have that   the Heisenberg vertex algebra  $M(1) $ is a subalgebra of $V$. By using the fact that   each  $F_{\ell}$ is irreducible $M(1)$--module with highest weight vector $e^{\ell \alpha}$ we see that  ${\mathfrak Z}(V)$   is contained in  $M_0$. The proof follows.
   \end{proof}
   We have the following result:
   
   \begin{theorem}
   \item[(1)] The vertex algebra $V$ is strongly generated by
   $$ \{ E_{i,j} \ \vert \ i,j=1,2 \}$$
   where
   \bea
   E_{1,1}&=& \alpha \label{def11} \\
   E_{1,2}&=& \Psi^+ (-\tfrac{1}{2} ) a^- (-\tfrac{1}{2} ) {\bf 1}  \label{def12} \\
   E_{2,1}&=& \Psi^- (-\tfrac{1}{2} ) a^+ (-\tfrac{1}{2} ) {\bf 1}  \label{def21} \\
 E_{2,2}&=&a^+ (-\tfrac{1}{2} ) a^- (-\tfrac{1}{2} ) {\bf 1} -  \alpha \label{def22}
   \eea
   
      \item[(2)] The vertex algebra $V$ is isomorphic to $V^{cri}(\frak{gl} (1|1))$.
      
   \item[(3)] 
   The center of $V$ is isomorphic to the commutative vertex algebra $M_0$. Its Hilbert-Poincare series is given by (\ref{HP2}).

   \end{theorem}
\begin{proof}
The proof of assertion (1) is similar to \cite{K}.
Since  $n \ge 0$ we have
\bea && E_{1,1} (n) E_{1,2} = \delta_{n,0} E_{1,2}, \ \   E_{1,1} (n) E_{2,1} = -\delta_{n,0} E_{2,1}, \ \   E_{1,1} (n) E_{2,2} =- \delta_{n,1},    \nonumber \\
 &&   E_{2,2} (n) E_{1,2} = - \delta_{n,0} E_{1,2}, \ \   E_{2,2} (n) E_{2,1} = \delta_{n,0} E_{2,1}, \ \nonumber \\
 &&  E_{1,2} (n) E_{2,1} = \delta_{n,0} ( E_{1,1} + E_{2,2} ) \nonumber 
 \eea
 the commutator formula for vertex algebras directly implies that the components of the fields
 $$ E_{i,j} (z) = Y (E_{i,j}, z) = \sum_{n \in \Z} E_{i,j} (n) z ^ {-n-1} \qquad i,j \in \{1,2\} $$
 satisfy the commutation relations   (\ref{com}) for the affine ${\frak gl}(1 \vert 1)$ at the critical level. This gives a vertex algebra homomorphism 
 $\Phi : V^{cri}({\frak gl}(1|1)) \rightarrow V. $

Let $U = \mbox{Im} (\Phi)  =  V^{cri}({\frak gl}(1|1)). {\bf 1}$, i.e., $U$ is   the vertex subalgebra of $V$ generated by elements (\ref{def11})--(\ref{def22}). Consider 
$$ \widehat{U}= \bigoplus_{s \in \Z} V^{cri}({\frak gl}(1|1)) . e ^ {s \alpha}. $$
As in \cite{K} we see that each  $ V^{cri}({\frak gl}(1|1)) . e ^ {s \alpha}$ is obtained from $U$ by applying a   spectral flow automorphism for $\widehat{ {\frak gl}(1|1)}$, and therefore
 $$  (V^{cri}({\frak gl}(1|1)) . e ^ {s \alpha} )  \cdot  (V^{cri}({\frak gl}(1|1)) . e ^ {r \alpha} )   \subset    V^{cri}({\frak gl}(1|1)) . e ^ { (r+s) \alpha}  \qquad (r,s \in {\Z}). $$
This implies that  $\widehat{U}$ is a vertex subalgebra of $F \otimes M$. Since 
$$  a^{+ }  (-\tfrac{1}{2} ) {\bf 1}  = E_{2,1} (0) \Psi ^{+} (-\tfrac{1}{2}) {\bf 1},  \quad a^{- }  (-\tfrac{1}{2} ) {\bf 1}  = E_{1,2} (0) \Psi ^{-} (-\tfrac{1}{2}) {\bf 1}$$
we get 
$$ e^{\pm \alpha} , a^{\pm}  (-\tfrac{1}{2} ) {\bf 1}    \in  V^{cri}({\frak gl}(1|1))  . e ^ {\pm \alpha} \subset \widehat{U}. $$
So  all generators of $F \otimes M$ belong to the vertex subalgebra $\widehat{U}$.
Therefore $F \otimes M =   \widehat{U}$, 
which proves that $U =V$. This proves   (1).
 
 This gives a surjective vertex algebra homomorphism 
 $$\Phi : V^{cri}({\frak gl}(1|1)) \rightarrow V. $$
 The injectivity can be proved easily by using the PBW basis of $V^{cri}({\frak  gl}(1|1))$, which proves (2).
 The assertion (3)  follows from Lemma \ref{lema-center}.
 The proof follows.
 
\end{proof}

\begin{remark} 
It is interesting that the Hilbert-Poincare series of the center of  $V^{cri}({\frak gl}(1|1))$ coincides with $\mbox{ch}[\mathcal{W}_{1+\infty,-1}](q)$.  It is a natural question to see if there is a similar interpretation of the Hilbert-Poincare series for the center of $V^{cri}({\frak gl}(n|n))$  for general $n$. A different type of generalization will be proposed in Section \ref{generalization}. 
\end{remark}

\section{Modules for $V^{crit}({\frak gl}(1|1)) $ of the Whittaker type}
\label{Whittaker}

Affine vertex algebras  at critical levels also contain  modules of the Whittaker type. In the case of the affine Lie algebra $A_1^{(1)}$ such modules were constructed by the author, R. Lu and K. Zhao in \cite{ALZ}. In this section we show that our realization from Section    \ref{new-real} can be applied in a construction  of the  Whittaker type of modules  for $V^{cri}({\frak gl}(1|1)) $.

\begin{theorem}
For every  $\chi^+, \chi^- \in {\C}((z)) $,  $\chi ^{\pm} \ne 0$, $F(\chi^+, \chi^{-}) := F \otimes M(\chi^+, \chi^-)$ is an irreducible $V^{cri}(\frak{gl}(1|1)) $--module.
\end{theorem}
\begin{proof}
Since $M(\chi^+, \chi^-)$  is a  $M$--module, we have that $F \otimes M(\chi^+, \chi^-)$ is a $F\otimes M$--module, and therefore $F(\chi^+, \chi^{-}) $  is a $(F\otimes M)_0 = V^{cri}({\frak gl}(1|1))  $--module. Since $M(\chi^+, \chi^-)$ is $1$-dimensional, it follows that as a vector space $F(\chi^+, \chi^{-}) $ is isomorphic to $F$. 

We first prove that  $F(\chi^+, \chi^{-})$ is  a cyclic  $V^{cri}({\frak gl}(1|1)) $--module and  $F(\chi^+, \chi^{-}) = V^{cri}({\frak gl}(1|1))  . {\bf 1} $. 
Since  $\chi ^{\pm} \ne 0$,  there are $p^{\pm} \in {\Z} $ such that
$$\chi^{\pm} (z) = \sum _{ j  \ge -p ^{\pm} } \chi ^{\pm}  _{-j} z^{j-1} \quad  \chi ^{\pm}  _{p^{\pm} }  \ne 0. $$
The action of $E_{1,2}(n)$ and $E_{2,1}(n)$ are given by the following formulas:
\bea
E_{1,2} (n) &=& \sum_{j \ge - p^-} \chi^- _{-j} \Psi^ + (n-1/2 +j), \label{for-1} \\
E_{2,1} (n) &=& \sum_{j \ge - p^+} \chi^+ _{-j} \Psi^ - (n-1/2 +j). \label{for-2}
\eea
Relations (\ref{for-1}) -(\ref{for-2})  imply   that for $m \in {\Z}_{> 0}$
\bea  E_{1,2} (p^- -m+ 1 ) \cdots E_{1,2} (p^- ) {\bf 1} &=&  (\chi^- _{p^-} )^m \Psi^+ (-m-\tfrac{1}{2}) \cdots  \Psi^+ (-\tfrac{1}{2}) {\bf 1} \nonumber \\
&=& \nu _1 e^{ m  \alpha} \quad (\nu_1 \ne 0), \nonumber \\
 E_{2,1} (p^+ -m+1 ) \cdots E_{2,1} (p^+ ) {\bf 1} &=&  (\chi^+ _{p^+} )^m \Psi^- (-m-\tfrac{1}{2}) \cdots  \Psi^- (-\tfrac{1}{2}) {\bf 1}    \nonumber    \\
 &=& \nu _2 e^{-  m  \alpha} \quad (\nu_2 \ne 0). \nonumber
\eea
This proves that 
$e^{\ell \alpha} \in   V^{cri}({\frak gl}(1|1))  . {\bf 1} $ for every $\ell \in {\Z}$. Since the Heisenberg vertex algebra $M(1)$ is a subalgebra of $V^{cri}({\frak gl}(1|1)) $, we get that $F_{\ell} \subset V^{cri}({\frak gl}(1|1))  . {\bf 1} $ for each $\ell$. Therefore $F(\chi^+, \chi^{-})$  is a  cyclic module.

Assume that $U $ is a non-trivial submodule in $F(\chi^+, \chi^{-})$. Using the action of the Heisenberg vertex algebra $M(1)$ we get  that there is $\ell_0 \in {\Z}$ such that
$ e^{\ell_0 \alpha} \in U$.
By using  the actions of elements  $E_{1,2}(n) $ and $E_{2,1}(n)$ and formulas (\ref{for-1}) -(\ref{for-2}),  we easily  get that $e^{\ell  \alpha} \in U $ for every $\ell  \in {\Z}$.  Finally,   by applying the action of the Heisenberg vertex algebra  $M(1)$  on vectors $e^{\ell  \alpha}$  we get that  $F_{\ell} \subset U$ for every $\ell  \in {\Z}$. The proof follows.
\end{proof}

\section{A generalization}
\label{generalization}

We shall briefly discuss a possible generalization of our construction. We omit some technical details, since  a more detailed analysis will be presented in our forthcoming papers.

Let $n \in {\N}$, and consider the vertex algebras
$ F^{(n)}$  and  $M^{(n)}$.
These vertex algebras admit a natural action of the Lie algebra ${\frak gl}_n$ (see previous sections).  So ${\frak gl}_n$  acts on $F^{(n)}  \otimes M^{(n)}$.  Define the vertex algebra $V_n$ as the fixed point subalgebra of  this action:
$$ V_n =  (F^{(n)}  \otimes M^{(n)} ) ^{{\frak gl}_n}. $$
Let ${\mathfrak Z}(V_n)$  denotes the center of the vertex algebra $V_n$.
\begin{proposition}
${\mathfrak Z}(V_n)$ is isomorphic to the vertex algebra $( M^{(n)}   )^{{\frak gl}_n}$.
\end{proposition}
\begin{proof}
Clearly $( M^{(n)}   )^{{\frak gl}_n} \subset {\mathfrak Z}(V_n) $.
Next we  notice that $F^{(n)}$ is a completely reducible $W_{1+\infty, n} \times {\frak gl}_n$--module. By using the explicit decomposition from  \cite{FKRW} and \cite{KR}  we see that
$$\{ v \in F^{(n)} \ \vert \ u_m v = 0 \ \forall u \in W_{1+\infty, n}, \ m \ge 0   \} = {\C} {\bf 1}. $$
 Since $W_{1+\infty, n}$ is a vertex subalgebra of $V_n$,  we  get that  ${\mathfrak Z}(V_n) \subset M \cap V_n =  ( M^{(n)}  )^{{\frak gl}_n}$.  The claim follows.
\end{proof}

\begin{proposition}
The vertex algebra $V_n$ is strongly generated by the following vectors
\bea
&& j^{0,k} = -\sum_{i=1}^n \Psi^+ _i (-1/2) \Psi^- _i (-k -1/2) {\bf 1}, \\
&& j^{1,k} = \sum_{i=1}^n a^+ _i (-1/2) a^- _i (-k -1/2) {\bf 1}, \\
&& j^{+,k} = -\sum_{i=1}^n \Psi^+ _i (-1/2) a^- _i (-k -1/2) {\bf 1}, \\
&& j^{-,k} = \sum_{i=1}^n a^+ _i (-1/2) \Psi ^- _i (-k -1/2) {\bf 1}, 
\eea
where $ 0\le k \le n-1$.

\end{proposition}
\begin{proof}
The proof is essentially the same as the proof of \cite[Theorem 7.1]{CL}. Here is the sketch of the proof with  the explanation of some    basic steps.
\begin{itemize}
\item  The vertex algebra $F^{(n)} \otimes M^{(n)}$ admits the filtration
$$ \mathcal {SM}_0 \subset  \mathcal {SM} _1 \subset \cdots , \quad F^{(n)} \otimes M^{(n)} = \cup _{k \ge 0} \mathcal {SM}_k ,  $$
where $\mathcal {SM} _k$ is spanned by the products of generators $\Psi_i ^{\pm} (-n-1/2)$, $a_ i ^{\pm} (-n-1/2)$ of length at most $k$.
Since  the action of ${\frak gl}_n$ preserves the above filtration on $F^{(n)} \otimes M^{(n)}$, we have the isomorphism of the associated graded (super)algebras:
$$ \mbox{gr} (V_n ) = \left ( \mbox{gr} ( F^{(n)} \otimes M^{(n)} ) \right) ^{\frak{gl}_n}  =   \left ( \mbox{gr} ( F^{(n)} \otimes W^{(n)} ) \right) ^{\frak{gl}_n}  . $$
The generators of the above $\Z_{\ge 0}$--graded (super) algebras are determined in \cite[Section 5]{CL}.
As a consequence we conclude that 
 $V_n$ is strongly generated by elements
$j^{0,k}, j^{1,k}, j ^{\pm, k}$, $k \ge 0$.

\item  By using the fact that $V_n$ contains the vertex subalgebra isomorphic to $W_{1+\infty,n}$, which is strongly generated by $j^{0,k}$, $1\le k \le n$, we get relation
$$ j^{0,m} = P_m(j ^{0,0}, \cdots j^{0,n-1} ) \quad  m\ge n$$
(cf. relation (7.1) of \cite{CL})).
 Then by  applying  the ${\frak gl}(1 \vert 1)$ action on the decoupling relation above, we get that $V_n$ is strongly generated by the generators described in the statement.
\end{itemize}

\end{proof}


\end{document}